\documentclass{amsart}
\usepackage{amssymb,graphicx}
\usepackage{mathrsfs}
\usepackage{color}
\usepackage{enumerate}
\usepackage[colorlinks=true,citecolor=blue,linkcolor=blue]{hyperref}

\newcommand{\Rl}{\mathbb{R}}
\newcommand{\Cplx}{\mathbb{C}}
\newcommand{\Itgr}{\mathbb{Z}}
\newcommand{\Ntrl}{\mathbb{N}}

\newcommand{\Bc}{\mathcal{B}}
\newcommand{\Cc}{\mathcal{C}}

\newcommand{\Ec}{\mathcal{E}}

\newcommand{\diag}{\mathrm{diag}}

\newcommand{\schur}{\circ}

\def\XXint#1#2#3{{\setbox0=\hbox{$#1{#2#3}{\int}$ }
\vcenter{\hbox{$#2#3$ }}\kern-.6\wd0}}

\numberwithin{equation}{section}

\newtheorem{theorem}{Theorem}[section]

\newtheorem{corollary}[theorem]{Corollary}
\newtheorem{definition}[theorem]{Definition}
\newtheorem{lemma}[theorem]{Lemma}

\theoremstyle{remark}

\newtheorem{remark}[theorem]{Remark}

\title[Patterns, submajorisation and operator Lipschitz functions]{Schur bounded patterns, submajorisation and operator Lipschitz functions}
\author[]{E. McDonald}
\address{LAMA, Universit\'e Paris-Est Cr\'eteil, France}
\email{ed@eamcdonald.xyz}
\author[]{F. Sukochev}
\address{School of Mathematics and Statistics, University of New South Wales, Australia}
\email{f.sukochev@unsw.edu.au}
\author[]{D. Zanin}
\address{School of Mathematics and Statistics, HNP-LAMA, Hunan Basic Science Research Center for Mathematical Analysis, Central South University, Changsha, Hunan 410083, PRC}
\email{d.zanin@csu.edu.cn}

\date{}

\begin{document}
\maketitle{}

\begin{abstract}
    A Schur bounded pattern is a subset $S\subset \Ntrl^2$ such that Schur multiplication by every bounded function on $\Ntrl^2$ supported on $S$ defines a bounded linear operator in the norm of $\mathcal{B}(\ell_2(\Ntrl)).$ Schur bounded patterns were characterised by Davidson-Donsig as being unions of row-bounded and column-bounded sets. We study the analogous question for sets $S$ such that element-wise multiplication by every bounded function on $S$ is bounded in ideals of compact operators that are not closed under submajorisation, in particular the Schatten ideals $\Cc_p$ with $0<p<1$ and the weak Schatten ideal $\Cc_{1,\infty}.$ Conversely we characterise the ideals that are not closed under submajorisation by their Schur bounded patterns. This has implications for the functions which are Lipschitz in the norm of ideals that are not closed under submajorisation. In particular such functions must be differentiable and have derivative that is asymptotically constant at infinity.
\end{abstract}

\section{Introduction}
Let $A = \{A_{j,k}\}_{j,k\geq 0}$ and $B = \{B_{j,k}\}_{j,k\geq 0}$ be infinite matrices of scalars with uniformly bounded entries (i.e., $A,B \in \ell_{\infty}(\Ntrl^2)$). The Schur product $A\schur B$ is the entrywise product
\[
    A\schur B = \{A_{j,k}B_{j,k}\}_{j,k\geq 0}.
\]
If $B$ defines a bounded operator on the Hilbert space $\ell_2(\Ntrl),$ we write its operator norm as $\|B\|_{\Bc},$ and we write $B \in \Bc(\ell_2(\Ntrl))$ or $B \in \Bc$ for brevity. We will freely identify operators $B \in \Bc$ with their matrix representation in the canonical basis $\{e_n\}_{n=0}^\infty.$

We will write $T_A$ for the operator of Schur multiplication by $A,$ i.e.
\[
    T_{A}(B) = A\schur B.
\]

We say that $A\in \ell_{\infty}(\Ntrl^2)$ is a bounded Schur multiplier (or just a Schur multiplier) if $T_A$ defines a continuous linear map from $\Bc$ to $\Bc.$ That is, if there exists a constant $C_A$ such that
\[
    \|A\schur B\|_{\Bc} \leq C_A\|B\|_{\Bc},\quad B \in \Bc(\ell_2(\Ntrl)).
\]
A subset $S$ of $\Ntrl^2$ is called a \emph{Schur bounded pattern} if every $A\in \ell_{\infty}(\Ntrl^2)$ supported on $S$ is a bounded Schur multiplier.

In this article, we shall consider Schur multipliers on (quasi)-normed ideals of compact operators in $\Bc(\ell_2(\Ntrl))$ \cite{Schatten, GK1, GK2,Simon,LSZ, DPS-Book}. A linear subspace
$(\Ec,\|\cdot\|_{\Ec})$ of bounded operators is a (two-sided) ideal if $ABC \in \Ec$ whenever $A,C \in \Bc(\ell_2(\Ntrl))$ and $B \in \Ec$. A quasi-norm $\|\cdot\|_{\Ec}$ on $\Ec$ satisfies the usual axioms of a norm on a vector space, save for the fact that we only have a quasi-triangle inequality
\[
    \|A+B\|_{\Ec} \leq C_{\Ec}(\|A\|_{\Ec}+\|B\|_{\Ec}),\quad A,B \in \Ec
\]
An ideal $\Ec$ equipped with a quasi-norm $\|\cdot\|_{\Ec})$ is called a quasi-normed ideal
$(\Ec,\|\cdot\|_{\Ec})$ if 
$\| ABC \|_{\Ec} \leq \| A \|_{\Bc}\|B \|_{\Ec} \| C\|_{\Bc}$
whenever $A,C \in \Bc(\ell_2(\Ntrl))$ and $B \in \Ec.$ We say that $(\Ec,\|\cdot\|_{\Ec})$ is quasi-Banach if the induced quasi-metric is complete.

%A similar question replaces the operator norm on $\Bc$ by some other unitarily invariant norm.

We will say that an operator $B \in \Bc(\ell_2(\Ntrl))$ is represented by a finite matrix if its matrix representation in the standard basis is finitely supported. We freely identify operators with their matrix representation, $B = \{B_{j,k}\}_{j,k\in \Ntrl}.$
\begin{definition}
    Let $(\Ec,\|\cdot\|_{\Ec})$ be a quasi-normed ideal of $\Bc(\ell_2(\Ntrl)).$ We say that $A\in \ell_{\infty}(\Ntrl^2)$ is a weak $\Ec$-Schur multiplier if
    \[
        \|A\schur B\|_{\Ec} \leq C_A\|B\|_{\Ec}
    \]
    for all operators $B$ represented by a finite matrix $\{B_{j,k}\}_{j,k\geq 0}.$
    We say that $A$ is a strong $\Ec$-Schur multiplier if
    \[
        \|A\schur B\|_{\Ec} \leq C_A\|B\|_{\Ec}
    \]
    for all operators $B \in \Ec.$

    A set $S\subset \Ntrl^2$ is called a weak (resp. strong) Schur $\Ec$-bounded pattern if any $A\in \ell_{\infty}(\Ntrl^2)$ supported on $S$ is a weak (resp. strong) $\Ec$-Schur multiplier.
\end{definition}
For most familiar operator ideals (e.g., the Schatten ideals $\Cc_p,$ the bounded operators $\Bc,$), there is no difference between strong and weak $\Ec$-Schur multipliers.
\begin{definition}
    A quasi-norm $\|\cdot\|_{\Ec}$ on an ideal $\Ec$ is said to have the Fatou property if the unit ball of $\Ec$ is closed in the strong operator topology. Equivalently, $(\Ec,\|\cdot\|_{\Ec})$ has the Fatou property if for any monotone increasing family of positive operators $\{A_n\}_{n\geq 0}$ with $\sup_{n\geq 0}\|A_n\|_{\Ec} < \infty,$ we $\sup_{n\geq 0} A_n\in \Ec$ and
    \[
        \|\sup_{n\geq 0} A_n\|_{\Ec} \leq \sup_{\geq 0} \|A_n\|_{\Ec}.
    \]
\end{definition}
For more details see \cite{DPS-Book}.
\begin{lemma}\label{schur_definitions_are_equivalent_for_fatou}
    Let $(\Ec,\|\cdot\|_{\Ec})$ be a quasi-normed ideal of $\Bc$ with the Fatou property. Then $A \in \ell_{\infty}(\Ntrl^2)$ is a strong $\Ec$-Schur multiplier if and only if it is a weak $\Ec$-Schur multiplier.
\end{lemma}
\begin{proof}
    Let $A$ be a weak $\Ec$-Schur multiplier.
    Let $B \in \Ec.$ For $N\geq 0,$ let $B_N$ be the truncation of $B$ to the corner $\{0,1,\ldots,N-1\}^2.$ Then $B_N\to B$ in the strong operator topology, and $\sup_{N\geq 0} \|B_N\|_{\Ec}<\infty.$ Since $A$ is weakly $\Ec$-Schur bounded, it follows that $\sup_{N\geq 0}\|T_A(B_N)\|_{\Ec} < \infty.$ Since
    \[
        T_A(B_N)\to T_A(B)
    \]
    also in the strong operator topology, by the Fatou property of $\Ec$ we conclude that $T_A(B)\in \Ec.$
\end{proof}
Henceforth we will exclusively be concerned with ideals which having the Fatou property, and therefore we will not distinguish between strong and weak Schur multipliers or patterns.

\begin{remark}
    If $A$ has the property that $T_A(\Ec)\subseteq \Ec,$ then Schur multiplication by $A$ is a bounded operator on $\Ec,$ i.e. $A$ is a strong $\Ec$-Schur multiplier. This follows from the closed graph theorem and the fact that Schur multiplication is continuous in the weak operator topology.
\end{remark}

\begin{remark}
    The full algebra $\Bc$ of bounded operators has the Fatou property. Therefore, strong $\Bc$-Schur multipliers and weak $\Bc$-Schur multipliers coincide with the traditional notion of ``bounded Schur multipliers", and Schur $\Bc$-bounded patterns in the strong and weak sense coincide with the traditional notion of Schur bounded patterns. Henceforth to avoid confusion we will always use the more precise terminology ``$\Bc$-Schur multiplier" and ``$\Bc$-bounded patterns".
\end{remark}
The problem of describing Schur $\Bc$-bounded patterns was suggested in \cite{NF} and then fully resolved by Davidson and Donsig \cite{DavidsonDonsig}. A Schur $\Bc$-bounded pattern must decompose into two sets,  one with a bound on the number of entries in each row, and the other with
a bound on the number of entries in each column. By duality, this result from \cite{DavidsonDonsig} immediately extends to the complete description of Schur $\Cc_1$-bounded patterns, where $\Cc_1$ is the trace class \cite{Simon}.

By standard interpolation theory, we infer that every Schur $\Cc_1$-bounded pattern is also a Schur $\Ec$-bounded pattern for any quasi-Banach ideal $(\Ec,\|\cdot\|_{\Ec})$, provided that the latter is an interpolation space for the Banach pair $(\Cc_1,  \Bc)$. We explain about this subclass of quasi-Banach ideals in a more detailed fashion below.

The Schur $\Cc_2$-bounded patterns consist of arbitrary subsets of $\Ntrl^2,$ while for $1<p\neq 2<\infty,$ the characterisation of Schur $\Cc_p$-bounded patterns is an open problem. We were motivated to understand what happens in the range $0<p<1.$ It turns out that this range is even simpler than $p=1,$ because when $p<1$ there are no non-trivial Schur $\Cc_p$-bounded patterns at all (Theorem \ref{pattern_theorem}).

Recall that the singular value sequence $\mu(A)$ of a compact operator $A\in \Bc(\ell_2(\Ntrl))$ is defined by
\[
    \mu(k,A) = \inf\{\|A-R\|_{\Bc} \;:\; \mathrm{rank}(R)\leq k\},\quad k\geq 0.
\]
Equivalently, $\mu(k,A)$ is the $(k+1)$th largest eigenvalue of $|A|,$ in non-increasing order with multiplicities.

Alternatively,  the class of quasi-Banach ideals $(\Ec,\|\cdot\|_{\Ec})$ may be described as those quasi-normed ideals whose quasi-norm satisfies the condition: if $\mu(A)\leq \mu(B)$ and $B \in \Ec,$ then $A\in \Ec$ and $\|A\|_{\Ec}\leq \|B\|_{\Ec}.$ This description also serves as a justification to the frequently used term \lq\lq symmetric\ operator ideals \rq\rq.

We write $A\prec\prec B$ if $\mu(A)\prec\prec \mu(B).$ Here, the submajorisation (in the sense of Hardy-Littlewood-Poly\`a) of sequences $x,y \in c_0(\Ntrl)$ are defined by
\[
    x\prec\prec y\Longleftrightarrow \sum_{k=0}^n x_k\leq \sum_{k=0}^{n} y_k,\text{ for all }n\geq 0
\]
and majorisation is defined as
\[
    x\prec y\Longleftrightarrow x\prec\prec y\text{ and } \sum_{k=0}^\infty x_k = \sum_{k=0}^\infty y_k.
\]

We say that $(\Ec,\|\cdot\|_{\Ec})$ is \emph{closed under submajorisation} if $A\prec\prec B$ and $B\in \Ec$ implies that $A \in \Ec.$ Up to a re-norming, this is equivalent to being a so-called fully symmetric ideal (see Section \ref{fully_symmetric_section} below).

% with $\|A\|_{\Ec} \leq \|B\|_{\Ec}.$ If $\Ec\subset \Bc$ is a quasi-normed ideal closed under submajorisation, then (up to an equivalent quasi-norm), $\Ec$ is fully symmetric (see Section \ref{fully_symmetric_section} below).

The noncommutative version of the fundamental theorem of Calder\`on and Mityagin \cite{Mityagin, Calderon} established in \cite[Section 2]{DDP} (see also \cite[Section 3.10]{DPS-Book}) precisely describes the class of all Banach ideals $(\Ec,\|\cdot\|_{\Ec})$ which are  interpolation spaces for the Banach pair $(\Cc_1,  \Bc)$ as the class of  fully symmetric ideals. The extension of this result to the class of quasi-Banach ideals is available from a recent paper \cite{CSZ} combined with verbatim arguments from \cite[Section 2]{DDP} or \cite[Section 3.10]{DPS-Book}.

In this article we present a complete  description of Schur $\Ec$-bounded patterns for any quasi-Banach ideal $(\Ec,\|\cdot\|_{\Ec})$ which fails to be fully symmetric. For numerous examples of such ideals, we refer to \cite{KS,Russu1,Russu3} and especially to \cite[Theorem 9]{SSS}.

%On the other hand, we say that $\Ec$ is closed under submajorisation if $A\prec\prec B$ and $B \in \Ec$ implies that $A\in \Ec.$ Being closed under submajorisation is strictly weaker than being fully symmetric.

\section{Statement of Main Result}
%
% For any quasi-normed ideal $\Ec,$ the set of Schur $\Ec$-bounded patterns is a bornology on $\Ntrl^2$ (that is, the set of complements of Schur bounded patterns is a filter). Some elementary properties are as follows:
Some immediate consequences of the definition of a Schur $\Ec$-bounded pattern are as follows:
\begin{lemma}\label{elementary_lemma}
    Let $(\Ec,\|\cdot\|_\Ec)$ be a quasi-normed ideal with the Fatou property, and let $S\subseteq \Ntrl^2.$
    \begin{enumerate}[{\rm (i)}]
        \item{}\label{permanence} If $S$ is a Schur $\Ec$-bounded pattern, and $S'\subset S,$ then $S'$ is a Schur $\Ec$-bounded pattern.
        \item{}\label{union} If $S,S'$ are Schur $\Ec$-bounded patterns, then $S\cup S'$ is a Schur $\Ec$-bounded pattern.
        \item{}\label{permutation} Let $a,b:\Ntrl\to \Ntrl.$ be injections. If $S$ is a Schur-bounded pattern, then so is $\{(a(n),b(m))\,:\,(n,m)\in S\}.$
        \item{}\label{finite_rank} For any finite set $F\subset \Ntrl,$ $F\times \Ntrl$ and $\Ntrl\times F$ are Schur $\Ec$-bounded patterns.
    \end{enumerate}
\end{lemma}
\begin{proof}
    Parts \eqref{permanence} and \eqref{union} are obvious consequences of linearity. To prove \eqref{permutation}, let $U_a(e_n) = e_{a(n)}$ and $U_b(e_n) = e_{b(n)}.$ Then $U_a$ and $U_b$ are partial isometries, and
    \[
        \|U_a^*AU_b\|_{\Ec} \leq \|A\|_{\Ec}.
    \]
    To prove \eqref{finite_rank}, observe that if $m$ is supported in $F\times \Ntrl,$ then $\{m(j,k)A_{j,k}\}_{j,k\geq 0}$ has rank at most $|F|,$ and hence belongs to any ideal of $\Bc$ (c.f. \cite[Corollary 2.3]{Simon}).
\end{proof}

We restate the main result of \cite{DavidsonDonsig} using our terminology:
\begin{theorem}\cite[Thorem 2.3]{DavidsonDonsig}
    Let $(\Ec,\|\cdot\|_{\Ec})$ be a quasi-Banach ideal which is closed under submajorisation, and let $S\subseteq \Ntrl^2$ be a union of a \emph{row-finite} and \emph{column-finite} set. That is, $S = R\cup C,$ where
    \[
        \sup_{n\geq 0} |R\cap (\Ntrl\times \{n\})| < \infty,\; \sup_{n\geq 0} |C\cap (\{n\}\times\Ntrl)| < \infty.
    \]
    Then $S$ is a Schur $\Ec$-bounded pattern.

    Conversely, if $\Ec=\Cc_1$ (the trace ideal) or $\Ec = \Bc$ (the full algebra of bounded operators), then every Schur $\Ec$-bounded pattern is of this form.
\end{theorem}

Lemma \ref{elementary_lemma} proves that for any $n\geq 0,$ the subsets of
\[
    (\{0,1,\ldots,n\}\times \Ntrl)\cup (\Ntrl\times \{0,1,\ldots,n\})
\]
are strong Schur $\Ec$-bounded patterns for \emph{every} quasi-normed ideal $\Ec \subseteq B(\ell_2).$ These are \emph{trivial} Schur bounded patterns. In this note our aim is to characterise the ideals for which there are \emph{non-trivial} Schur bounded patterns.
\begin{theorem}\label{pattern_theorem}
Let $\Ec$ be a symmetric quasi-Banach ideal of $\Bc$ with the Fatou property. The following are equivalent:
\begin{enumerate}[{\rm (i)}]
    \item{}\label{only_trivial} Every Schur $\Ec$-bounded pattern is a subset of a set of the form $(F\times \Ntrl)\cup (\Ntrl\times F)$ for some finite subset $F\subseteq \Ntrl,$
    \item{}\label{no_diagonal} The diagonal $\Delta = \{(n,n)\;:\;n\geq 0\}$ is not a Schur $\Ec$-bounded pattern.
    \item{}\label{no_toeplitz} The only Toeplitz Schur $\Ec$-bounded pattern is the empty set
    \item{}\label{no_hankel} The only Hankel Schur $\Ec$-bounded patterns are finitely supported.
    \item{}\label{not_fully} $(\Ec,\|\cdot\|_{\Ec})$ is not closed under submajorisation.
    %closed under submajorisation.
\end{enumerate}
\end{theorem}
In other words, an ideal admits non-trivial Schur bounded patterns if and only if it is closed under submajorisation. The proof of this Theorem is completed in Section \ref{final_result_section}.
A Toeplitz Schur $\Ec$-bounded pattern is a pattern of the form $\{(j,k)\;:\;j-k\in S_0\}$ where $S_0\subseteq\Itgr,$ while a Hankel Schur $\Ec$-bounded pattern is a pattern of the form $\{(j,k)\;:\;j+k\in S_1\}$ for some $S_1\subseteq \Ntrl.$ These notions were studied by Farforovskaya-Nikolskaya \cite{NF}.

% By Lemma \ref{elementary_lemma}.\eqref{finite_rank}, the subsets of sets of the form $F\times \Ntrl\cup \Ntrl\times F$ for $|F|<\infty$ are strong Schur $\Ec$-bounded for every quasi-Banach ideal $\Ec.$ These are ``trivial" Schur $\Ec$-bounded patterns. Theorem \ref{pattern_theorem} characterises those ideals which have no ``non-trivial" Schur $\Ec$-bounded pattern.

Theorem \ref{pattern_theorem} implies that for $0<p<1,$ the only Schur $\Cc_p$-bounded patterns are the trivial patterns in Lemma \ref{elementary_lemma}.\eqref{finite_rank}.

\subsection{Application to operator Lipschitz functions}
A function $f:\Rl\to\Cplx$ is said to be operator Lipschitz if there exists a constant $C_f$ such that for all bounded self-adjoint $A$ and $X$ we have
\[
    \|f(A+X)-f(A)\|_{\Bc} \leq C_f\|X\|_{\Bc}.
\]
To be operator Lipschitz is a condition slightly stronger than differentiability \cite[Theorem (4.1)]{JohnsonWilliams} and slightly weaker than $C^{1+\varepsilon}$ for all $\varepsilon>0$ \cite{Peller}.
More generally, we can speak of $\Ec$-Lipschitz functions:
\begin{definition}
Let $(\Ec,\|\cdot\|_{\Ec})$ be a quasi-Banach ideal of $\Bc.$ A function $f$ on $\Rl$ is said to be $\Ec$-Lipschitz if there exists a constant $C_f$ such that
\[
    \|f(A+X)-f(A)\|_{\Ec}\leq C_f\|X\|_{\Ec}
\]
for all bounded self-adjoint operators $A,$ and for all self-adjoint $X \in \Ec.$
\end{definition}
This notion has been studied by many authors, see e.g. \cite{Davies,PotapovSukochev, Peller, KissinShulman}. For $1<p<\infty,$ $\Cc_p$-Lipschitzness is identical to classical Lipschitzness \cite{PotapovSukochev}. There is a well-known relationship between $\Ec$-Lipschitz functions and $\Ec$-Schur multipliers, which we review in Appendix \ref{lowner_appendix}.

% In \cite{MS} it was proved that for $0<p<1,$ no amount of differentiability implies $\Cc_p$-Lipschitness, and in fact any periodic function which is $\Cc_p$-Lipschitz must be constant. The idea of the proof there was to show that the existence of a non-constant periodic $\Cc_p$-Lipschitz function implies the boundedness of the diagonal mapping $X\mapsto \mathrm{diag}(X)$ from $\Cc_p$ to $\ell_p.$
%
% The boundedness of the diagonal mapping is equivalent to the claim that the diagonal set $\{(n,n)\;:\; n\geq 0\}$ is a $\Cc_p$-Schur bounded pattern. Using Theorem \ref{pattern_theorem} we can strengthen the above-mentioned result of \cite{MS} to the following:

The following theorem strengthens \cite[Theorem 1.0.1]{MS}. In that paper, it was proved that no non-constant periodic function is Lipschitz in $\Cc_p$ for $0<p<1,$ in particular no amount of smoothness ensures $\Cc_p$-Lipschitzness. Here $\Cc_p$ can be replaced with an arbitrary ideal not closed under submajorisation, and periodicity is weakened to having non-vanishing derivative at infinity.
\begin{theorem}\label{periodic_functions_theorem}
    Let $(\Ec,\|\cdot\|_{\Ec})$ be a quasi-Banach ideal with the Fatou property. The following are equivalent:
    \begin{enumerate}[{\rm (i)}]
        \item{}\label{is_not_closed_periodic} $\Ec$ is closed under submajorisation
        \item{}\label{has_not_OL_periodic} Every operator Lipschitz function is $\Ec$-Lipschitz
        \item{}\label{has_vanishing_derivative_periodic} There exists a bounded differentiable $\Ec$-Lipschitz function $f$ such that $f'(\lambda)$ does not vanish as $\lambda\to \infty.$
    \end{enumerate}
\end{theorem}
We emphasise the equivalence of \eqref{has_not_OL_periodic} and \eqref{has_vanishing_derivative_periodic}: if $\Ec$ is not closed under submajorisation, then necessarily every bounded $\Ec$-Lipschitz function $f$ has $\lim_{\lambda\to \pm\infty} f'(\lambda)=0.$ Note that in order this to make sense, we must first show that if $\Ec$ is not closed under submajorisation then every $\Ec$-Lipschitz function is differentiable. This is Corollary \ref{not_fully_symm_implies_differentiable}.

\begin{remark}
As illustration, the weak-trace class ideal $\Cc_{1,\infty}$ with quasi-norm
\[
    \|A\|_{\Cc_{1,\infty}} := \sup_{k\geq 0} (k+1)\mu(k,A)
\]
is a quasi-Banach ideal with the Fatou property which is not closed under submajorisation. Theorem \ref{periodic_functions_theorem} implies that any function $f$ which satisfies
\[
    \|f(A)-f(B)\|_{\Cc_{1,\infty}} \leq C_f\|A-B\|_{\Cc_{1,\infty}}
\]
is necessarily differentiable, and has vanishing derivative at infinity. This is a new result.
\end{remark}

\section{A characterisation of ideals closed under submajorisation}\label{fully_symmetric_section}
Recall that we say that an ideal $\Ec\subset \Bc(\ell_2(\Ntrl))$ is closed under submajorisation if $A\in \Bc, B \in \Ec$ satisfy $A\prec\prec B$ then $A \in \Ec.$

A formally stronger notion is as follows: a quasi-Banach ideal $(\Ec,\|\cdot\|_{\Ec})$ is called \emph{fully symmetric} if $\Ec$ is closed under submajorisation and $A\prec\prec B$ implies $\|A\|_{\Ec}\leq \|B\|_{\Ec}$ \cite{DDP}.

There is not a significant difference between a quasi-Banach ideal being closed under submajorisation and being fully symmetric, as the next lemma shows. The same result was proved for Banach function space by Braverman and Mekler \cite[Proposition 2.1]{BravermanMekler1975}, and their proof can be adapted to this setting without difficulty.
\begin{lemma}\label{closed_under_submajorisation_implies_fully_symmetric}
    Let $(\Ec,\|\cdot\|_{\Ec})$ be a quasi-Banach ideal of $\Bc.$ If $\Ec$ is closed under submajorisation, then $\Ec$ admits a quasinorm $\|\cdot\|_{\Ec}'$ which is equivalent to $\|\cdot\|_{\Ec}$ and such that $(\Ec,\|\cdot\|_{\Ec}')$ is fully symmetric.
\end{lemma}
\begin{proof}
    Define
    \[
        \|B\|_{\Ec}' = \sup\{\|A\|_{\Ec}\;:\; A\prec\prec B\}.
    \]
    Obviously $\|B\|_{\Ec}\leq \|B\|_{\Ec'},$ although it is not obvious that $\|B\|_{\Ec}'<\infty.$ We argue that there exists a constant $C<\infty$ such that $\|B\|_{\Ec}' \leq C\|B\|_{\Ec}.$ Indeed, otherwise we may select sequences $\{A_n\}_{n=0}^\infty$ and $\{B_n\}_{n=0}^\infty$ of positive operators such that $A_n\prec\prec B_n,$ $\|A_n\|_{\Ec} = n,$ $\|B_n\|_{\Ec} \leq 2^{-n}.$

    Ky-Fan's inequalities (c.f. \cite[Theorem 3.9.9]{DPS-Book}) imply that
    \[
        \sum_{n=0}^N A_n\prec\prec \sum_{n=0}^N \mu(B_n).
    \]
    Note that since $\|B_n\|_{\Ec} \leq 2^{-n},$ it follows that the sum $\sum_{n=0}^\infty \mathrm{diag}(\mu(B_n))$ converges in the norm of $\Ec.$ This is true in the quasi-Banach setting, by the Aoki-Rolewicz theorem \cite[Theorem 1.3]{KPR}. Since $A_n$ are positive, the continuity of the singular value function under monotone convergence implies that
    \[
        \sum_{n=0}^{\infty} A_n\prec\prec \sum_{n=0}^{\infty} \mathrm{diag}(\mu(B_n)).
    \]
    Since $\Ec$ is closed under submajorisation, it follows that $\sum_{n=0}^{\infty} A_n\in \Ec.$ However, since $\Ec$ is symmetric we have
    \[
        N = \|A_N\|_{\Ec} \leq \|\sum_{n=0}^N A_n\|_{\Ec}
    \]
    and so the sum $\sum_{n=0}^\infty A_n$ cannot converge in $\Ec.$
    Thus there exists $C<\infty$ such that
    \[
        \|B\|_{\Ec} \leq \|B\|_{\Ec}' \leq C\|B\|_{\Ec}.
    \]
    Obviously we have $\|\lambda B\|_{\Ec}' = |\lambda| \|B\|_{\Ec}',$ and if $A\prec\prec B$ then $\|A\|_{\Ec}'\prec\prec \|B\|_{\Ec}'.$ The only remaining property to check is the quasi-triangle inequality.

    Let $B_1,B_2\in \Ec.$ For $\varepsilon>0,$ there exist $A\in \Ec$ such that $A\prec\prec B_1+B_2$ and $\|B_1+B_2\|'\leq \|A\|_{\Ec}+\varepsilon.$ Since $A\prec\prec B_1+B_2,$ there exists a linear operator $T:\Bc\to \Bc,$ contractive in both the trace class and the operator norm, such that
    $A = T(B_1+B_2),$ see \cite[Theorem 2.2(i)]{DDP}. Since $T$ is contractive in the operator and trace-class norms, we have $T(B_1)\prec\prec B_1$ and $T(B_2)\prec\prec B_2.$ The quasi-triangle inequality in $\Ec$ implies
    \[
        \|A\|_{\Ec}\leq C_{\Ec}(\|T(B_1)\|_{\Ec}+\|T(B_2)\|_{\Ec})\leq C_{\Ec}(\|B_1\|_{\Ec}'+\|B_2\|_{\Ec}')
    \]
    and therefore
    \[
        \|B_1+B_2\|_{\Ec}' \leq C_{\Ec}(\|B_1\|_{\Ec}'+\|B_2\|_{\Ec}')+\varepsilon.
    \]
    Since $\varepsilon$ is arbitrary, this completes the proof that $\|\cdot\|_{\Ec}'$ is a quasi-norm.
\end{proof}

\section{The relationship between the diagonal pattern and submajorisation}
Let $\Delta = \{(n,n)\}_{n\geq 0} \subset \Ntrl^2.$ In this section we prove the equivalence of \eqref{no_diagonal} and \eqref{not_fully} in Theorem \ref{pattern_theorem}:
\begin{theorem}\label{boundedness_of_diagonal}
    Let $(\Ec,\|\cdot\|_{\Ec})$ be a symmetric quasi-normed ideal of $\Bc$ with the Fatou property. The diagonal $\Delta$ is a Schur $\Ec$-bounded pattern if and only if $\Ec$ is fully symmetric.
    
%    closed with respect to Hardy-Littlewood majorisation.
\end{theorem}
This result proof essentially amounts to the Schur-Horn theorem, as we describe below.

If $V$ is a matrix, let $\mathrm{diag}(V)$ be the diagonal sequence $\{V_{j,j}\}_{j\geq 0}.$ We have
\begin{equation}\label{diagonal_restriction_is_bounded}
    \diag(V) \prec\prec \mu(V).
\end{equation}
A quick way to prove \eqref{diagonal_restriction_is_bounded} is via the identity
\[
    \mathrm{diag}(A) = \int_0^{1} U(t)AU(t)^*\,dt
\]
where $U(t) = \{\exp(2\pi i tj)\delta_{j,k}\}_{j,k\geq 0},$ and where the integral converges entrywise. This implies that for a finite matrix $A,$ we have
\[
    \|\mathrm{diag}(A)\|_{\Cc_1} \leq \int_0^1 \|A\|_{\Cc_1}\,dt = \|A\|_{\Cc_1}
\]
and
\[
    \|\mathrm{diag}(A)\|_{\Bc} \leq \int_{0}^1 \|A\|_{\Bc}\,dt = \|A\|_{\Bc}
\]
which together imply \eqref{diagonal_restriction_is_bounded}.

The following lemma is a weaker version of the Kaftal-Weiss theorem \cite{KaftalWeiss}, but it much easier to prove.
\begin{lemma}\label{elementary kaftal-weiss}
Given non-negative sequences $x,y\in c_0(\Ntrl)$ with $y\prec\prec x,$ there exists a positive $V=\{a_{j,k}\}_{j,k\geq0}\in \Bc$ such that $\diag(V)=y$ and $\mu(V)\leq\mu(x).$
\end{lemma}
\begin{proof}
    Let $n\geq 1,$ and $y_n:=y\chi_{[0,n)}$ and choose $0\leq x_n\leq x\chi_{[0,n)}$ such that $y_n\prec x_n.$ By the Schur-Horn theorem (c.f. \cite[Theorem 10.18]{Zhang-matrix-theory-2011}) there exists a positive matrix $V_n\in M_n(\mathbb{C})$ with $\mathrm{diag}(V_n) = y_n$ and $\mu(V_n)=\mu(x_n).$ We identify $M_n(\mathbb{C})$ with a corner in $\Bc$ and, therefore, view $V_n$ as an element of $\Bc.$ The sequence $\{V_n\}_{n\geq0}$ is bounded and is therefore pre-compact in the weak operator topology. Choose a cluster point $V$ of the sequence $\{V_n\}_{n\geq0}$ and a convergent subnet of the sequence $\{V_n\}_{n\geq0}$ converging to $V.$ The diagonal of $V$ is $y.$ By \cite[Lemma 2.2.13]{LSZ}, $\mu(V)\leq\mu(x).$
\end{proof}

\begin{proof}[Proof of Theorem \ref{boundedness_of_diagonal}]
    Let $M\in \ell_{\infty}(\Ntrl^2)$ be supported on $\Delta.$ By \eqref{diagonal_restriction_is_bounded}, we have
    \[
        \mu(T_MA) \leq \|M\|_{\infty}\mu(\mathrm{diag}(A)) \prec\prec \|M\|_{\infty}\mu(A).
    \]
    It follows that if $\Ec$ is fully symmetric,
% closed under submajorisation, 
then $T_M$ is bounded on $\Ec.$ It follows that
$\Delta$ is a Schur bounded pattern.

    Conversely, suppose that $\Ec$ is not fully symmetric.
 %closed under submajorisation. 
 This means that there exists $B\in \Ec$  and $A\notin \Ec$ such that $\mu(A) \prec\prec \mu(B).$ By Lemma \ref{elementary kaftal-weiss}, there exists $V\in \Bc$ such that $\mathrm{diag}(V) = \mu(A),$ and $\mu(V)\leq \mu(B).$ Since $\Ec$ is symmetric, we have $V \in \Ec,$ but $\mathrm{diag}(V)\notin \Ec.$
\end{proof}

\section{The existence of nontrivial Schur bounded patterns is equivalent to closure under submajorisation}\label{final_result_section}
We are now prepared to prove our first main result.
\begin{proof}[Proof of Theorem \ref{pattern_theorem}] The diagonal is not a subset of a set of the form $(F\times \Ntrl)\cup( \Ntrl\times F)$ for finite $F.$ Hence, \eqref{only_trivial} implies \eqref{no_diagonal}.

Conversely, suppose that there exists a Schur $\Ec$-bounded pattern $S$ such that
$$S\cap (\{n,n+1,n+2,\ldots\}\times \{n,n+1,n+2,\ldots\}) \neq \emptyset$$
for infinitely many $n.$ If follows that there exist strictly increasing sequences $\{n_k\}_{k\geq 0}$ and $\{m_k\}_{k\geq 0}$ such that $(n_k,m_k)_{k\geq0}\subset S.$ By Lemma \ref{elementary_lemma}.\eqref{permanence} and \eqref{permutation}, $\Delta$ is a Schur $\Ec$-bounded pattern.

The fact that \eqref{no_diagonal} is equivalent to \eqref{not_fully} is precisely Theorem \ref{boundedness_of_diagonal}.

Obviously, \eqref{only_trivial} implies \eqref{no_hankel} and \eqref{no_toeplitz}. Since any non-empty Toeplitz pattern contains an infinite diagonal, we also have the equivalence \eqref{no_toeplitz}$\leftrightarrow$\eqref{no_diagonal}.

Nikolskaya-Farforovskaya \cite[Lemma 3.6]{NF} proved that there exist infinite Hankel Schur $\Bc$-bounded pattern, which by duality must also be Schur $\Cc_1$-bounded pattern. Hence, if $\Ec$ is fully symmetric, then there exists an infinite Hankel Schur $\Ec$-bounded pattern, and this shows that \eqref{no_hankel} implies \eqref{not_fully}, which completes the proof.
\end{proof}

\section{Application to Lipschitz estimates}\label{lipschitz_section}
Recall that $\Bc$ denotes the bounded linear operators on $\ell_2(\Ntrl),$ and we identify $X \in \Bc$ with its matrix representation in the standard basis $\{e_n\}_{n=0}^\infty$ by $X_{j,k} = \langle e_j,Xe_k\rangle.$ Recall also that for a bounded function $m$ on $\Ntrl^2,$ $T_m$ denotes the Schur multiplier
\[
    (T_mX)_{j,k} = m(j,k)X_{j,k},\quad j,k\geq 0
\]
and
\[
    \|m\|_{M(\Ec)} := \|T_m\|_{\Ec\to \Ec}.
\]
Given a Lipschitz function $f$ on $\Rl,$ denote
\[
    f^{[1]}(\lambda,\mu) = \begin{cases} \frac{f(\lambda)-f(\mu)}{\lambda-\mu}&\quad \lambda \neq \mu,\\
    0&\quad \lambda=\mu
    \end{cases},\quad \lambda,\mu\in \Rl.
\]

% \begin{definition}\label{continuous_schur_definition}
%     Let $(\Ec,\|\cdot\|_{\Ec})$ be a quasi-Banach ideal of $\Bc.$ Let $A:\Rl^2\to \Cplx.$ We say that $A$ is a bounded Schur multiplier of $\Ec$ (briefly, an $\Ec$-Schur multiplier) if
%     \[
%         \|A\|_{M(\Ec)} := \sup_{\{\lambda_j\},\{\mu_k\},X} \frac{\|\{A(\lambda_j,\mu_k)X_{j,k}\}_{j,k\geq 0}\|_{\Ec}}{\|X\|_{\Ec}} < \infty
%     \]
%     where the supremum is taken over all \emph{disjoint} sequences $\{\lambda_j\}_{j=0}^\infty$ and $\{\mu_k\}_{k=0}^\infty$ in $\Rl$ and all finite matrices $X = \{X_{j,k}\}_{j,k\geq 0}.$
% \end{definition}
%
% \begin{remark}
%     The requirement that the supremum in Definition \ref{continuous_schur_definition} be taken over disjoint sequences is important in the quasi-Banach case.
% \end{remark}
%
% \begin{remark}
%     The sequences in Definition \eqref{continuous_schur_definition} allow repetitions, and so it would be more correct to call $A$ a \emph{completely} bounded Schur multiplier of $\Ec$ if $\|A\|_{M(\Ec)}< \infty.$
% \end{remark}

The following lemma is well-known for the case $\Ec = \Bc,$ and is proved in the general case in Appendix \ref{lowner_appendix}.
\begin{lemma}\label{Lipschitz_is_the_same_as_Schur_condition}
    Let $(\Ec,\|\cdot\|_{\Ec})$ be a quasi-normed ideal of $\Bc$ with the Fatou property. Then $f$ is $\Ec$-Lipschitz if and only if
    \[
        \sup_{\{\lambda_j\},\{\mu_k\}} \|\{f^{[1]}(\lambda_j,\mu_k)\}_{j,k\geq 0}\|_{M(\Ec)} < \infty.
    \]
    where the supremum is taken over all sequences $\{\lambda_j,\mu_k\}_{j,k=0}^\infty$ which are non-intersecting.
\end{lemma}
Note that the supremum in Lemma \ref{Lipschitz_is_the_same_as_Schur_condition} is over sequences which can have repetitions, meaning that if $f$ is $\Ec$-Lipschitz then every matrix $\{f^{[1]}(\lambda_j,\mu_k)\}_{j,k=0}^\infty$ is a completely bounded Schur multiplier of $\Ec.$

For this section, let
\[
    \Gamma(j,k) := \chi_{[0,\infty)}(j-k),\quad j,k\in \Ntrl.
\]
The Schur multiplier $T_{\Gamma}$ is the upper triangular truncation on matrices.

% \begin{lemma}\label{for_triangle_integers_are_enough}
%     Let $(\Ec,\|\cdot\|_{\Ec})$ be a quasi-Banach ideal of $\Bc.$ The triangular truncation $\Gamma$ is an $\Ec$-Schur multiplier if and only if $T_{\widetilde{\Gamma}}$ is completely bounded in $\Ec.$
% \end{lemma}
% \begin{proof}
%     {\color{red}
%         Still unsure how to do this.
%     }
% \end{proof}

\begin{lemma}\label{triangular_truncation_bounded_implies_fully_symmetric}
    Let $(\Ec,\|\cdot\|_{\Ec})$ be a quasi-Banach ideal of $\Bc$ with the Fatou property. If the triangular truncation $\Gamma$ is an $\Ec$-Schur multiplier, then $\Ec$ is closed under submajorisation.
\end{lemma}
\begin{proof}
    Recall that $\Delta = \{(n,n)\;:\; n\geq 0\}$ denotes the diagonal. For any finite matrix $X =\{X_{j,k}\}_{j,k\geq 0} \in \Ec,$ we have
    \[
        T_{\chi_{\Delta}}X = \{(\Gamma(j,k+1/2)-\Gamma(j+1,k+1/2))X_{j,k}\}_{j,k\geq 0}.
    \]

% {\color{blue} something like the above is true, but the the above equality per se seems false. here is verification.
% $$S=\sum_{j\geq0}E_{j,j+1},\quad S^{\ast}=\sum_{j\geq0}E_{j+1,j}=\sum_{j\geq1}E_{j,j-1},\quad X=\sum_{j,k\geq0}X_{j,k}E_{j,k}.$$
% We have
% $$XS^{\ast}=\sum_{j,k\geq0}X_{j,k}E_{j,k}\cdot\sum_{l\geq1}E_{l,l-1}=\sum_{\substack{j\geq0\\ l\geq1}}X_{j,l}E_{j,l-1}=\sum_{j,l\geq0}X_{j,l+1}E_{j,l}.$$
% Thus,
% $$\tilde{\Gamma}\circ (XS^{\ast})=\sum_{j\geq l\geq0}X_{j,l+1}E_{j,l}.$$
% Next,
% $$(\tilde{\Gamma}\circ (XS^{\ast}))\cdot S=\sum_{j\geq l\geq0}X_{j,l+1}E_{j,l}\cdot\sum_{k\geq0}E_{k,k+1}=\sum_{j\geq k\geq 0}X_{j,k+1}E_{j,k+1}=$$
% $$=\sum_{\substack{m\geq1\\ j\geq m-1}}X_{j,m}E_{j,m}.$$
% Hence,
% $$(\tilde{\Gamma}\circ (XS^{\ast}))\cdot S-\tilde{\Gamma}\circ X=\sum_{\substack{m\geq1\\ j\geq m-1}}X_{j,m}E_{j,m}-\sum_{\substack{m\geq0\\ j\geq m}}X_{j,m}E_{j,m}=$$
% $$=-\sum_{j\geq0}X_{j,0}E_{j,0}+\sum_{\substack{m\geq1\\ j=m-1}}X_{j,m}E_{j,m}=$$
% $$=-X\cdot E_{0,0}+\sum_{m\geq1}X_{m-1,m}E_{m-1,m}.$$
% The first summand is inessential (such a linear map is always bounded; or we can even kill it by making the shift bilateral). The second summand is approximately what we want --- but not exactly.
% }

    Hence by the quasi-triangle inequality, Schur multiplication by $\chi_{\Delta}$ is bounded on $\Ec.$
    It follows that $\Delta$ is a Schur $\Ec$-bounded pattern. By Theorem \ref{pattern_theorem}, $\Ec$ is therefore closed under submajorisation.
\end{proof}
The converse of Lemma \ref{triangular_truncation_bounded_implies_fully_symmetric} does not hold, for instance $\Cc_1$ is closed under submajorisation but Schur multiplication by $\Gamma$ is not bounded on $\Cc_1.$
%
% \begin{lemma}\label{triangle_matrix_is_bounded}
%     The function
%     \[
%         T(\lambda,\mu) := \frac{|\lambda+\mu|}{|\lambda|+|\mu|},\quad \lambda,\mu\in \Rl,\quad (\lambda,\mu)\neq (0,0)
%     \]
%     with $T(0,0) =1$ is a $\Bc$-Schur multiplier.
% \end{lemma}
% \begin{proof}
%     Let
%     \[
%         g(\lambda,\mu) = \frac{2\min\{\lambda,\mu\}}{\lambda+\mu},\quad \lambda,\mu \geq 0.
%     \]
%     This function is positive definite \cite[Example 3.6(c)]{HiaiKosaki}. Since $g$ is also bounded, it follows that $g$ is a $\Bc$-Schur multiplier \cite[Proposition 3.10]{HiaiKosaki}. Using the formula
%     \[
%         2\min\{\lambda,\mu\} = \lambda+\mu-|\lambda-\mu|
%     \]
%     we have
%     \[
%         1-g(\lambda,\mu) = \frac{|\lambda-\mu|}{\lambda+\mu},\quad \lambda,\mu>0
%     \]
%     and therefore
%     \begin{align*}
%         T(\lambda,\mu) &= \chi_{(-\infty,0]}(\lambda)\chi_{(-\infty,0]}(\mu)+\chi_{(0,\infty)}(\lambda)\chi_{(0,\infty)}(\mu)\\
%                        &\quad+ \chi_{(-\infty,0)}(\lambda)\chi_{ (0,\infty)}(\mu)(1-g(-\lambda,\mu))+ \chi_{(0,\infty)}(\lambda)\chi_{(-\infty,0)}(\mu)(1-g(\lambda,-\mu)).
%     \end{align*}
%     Schur multiplication by the functions
%     \[
%         \chi_{(-\infty,0]}(\lambda),\quad \chi_{(-\infty,0]}(\mu)
%     \]
%     correspond to left (resp. right) multiplication by a projection, and hence these are contractive Schur multipliers in $\Bc.$ In this way we can see that $T$ is a composition of bounded Schur multipliers.
% \end{proof}

\begin{remark}\label{componentwise_remark}
In finite dimensions, all notions of convergence are equivalent. If a sequence $\{A_n\}_{n\geq 0}\subset M_N(\mathbb{C})$ converges component-wise to $A\in M_N(\mathbb{C}),$ then $\|A_n\|_{\Ec}\to\|A\|_{\Ec}$ for every quasi-normed ideal $(\Ec,\|\cdot\|_{\Ec}).$ Equivalently, if $\{A_n\}_{n\geq 0}\subset \Ec$ is a sequence of matrices supported in the top left corner $\{0,\ldots,N-1\}^{2}$ which converge component-wise to $A\in \Ec,$ then $\|A_n\|_{\Ec}\to \|A\|_{\Ec}.$
\end{remark}

The proof of Lemma \ref{non_differentiable_implies_triangle} below follows the argument given in Johnson-Williams \cite[Theorem (4.1)]{JohnsonWilliams} (for the special case $\Ec=\mathcal{B}$).
The argument uses the following fact:
\begin{lemma}\label{injection_lemma}
If $\{x_{j,k}\}_{j,k=0}^\infty$ is a doubly-indexed sequence of real numbers such that
\[
    \lim_{j\to\infty} \lim_{k\to\infty} x_{j,k} = 1,\quad \lim_{k\to\infty}\lim_{j\to\infty} x_{j,k} = 0
\]
then for any $\varepsilon>0$ there exist injections $\sigma,\pi:\Ntrl\to\Ntrl$ such that
\[
    |x_{\sigma(j),\pi(k)}-\Gamma(j,k)| < \varepsilon,\quad j,k\geq 0.
\]
\end{lemma}
\begin{proof}
For $j,k\geq 0,$ let
\[
    a_j := \lim_{k\to\infty} x_{j,k},\quad b_k := \lim_{j\to\infty} x_{j,k}.
\]
We first show that there exist injections $\alpha,\beta:\Ntrl\to\Ntrl$ such that
\begin{align*}
    &|x_{\alpha(j),\beta(k)}-a_{\alpha(j)}| < \varepsilon,\quad j\leq k,
    &|x_{\alpha(j),\beta(k)}-b_{\beta(k)}| < \varepsilon,\quad j>k.
\end{align*}
Let $\alpha(-1)=\beta(-1)=0.$ We define $\alpha(0),\beta(0),\alpha(1),\beta(1),\ldots$ inductively by
\begin{align*}
    \alpha(j) &:= \min\{m> \alpha(j-1)\;:\; \max_{-1\leq \ell<j} |x_{m,\beta(\ell)}-b_{\beta(\ell)}| <\varepsilon\}\\
    \beta(j)  &:= \min\{m>\beta(j-1)\;:\; \max_{-1\leq \ell\leq j} |x_{\alpha(\ell),m}-a_{\alpha(\ell)}| < \varepsilon\}.
\end{align*}
Then, by construction, if $j>\ell,$ we have
\[
    |x_{\alpha(j),\beta(\ell)}-b_{\beta(\ell)}|< \varepsilon
\]
and if $\ell\leq j,$ we have
\[
    |x_{\alpha(\ell),\beta(j)}-a_{\alpha(\ell)}| < \varepsilon.
\]
In other words,
$$|x_{\alpha(j),\beta(k)}-b_{\beta(k)}|<\varepsilon,\quad j>k,$$
$$|x_{\alpha(j),\beta(k)}-a_{\alpha(j)}|<\varepsilon,\quad j\leq k.$$

With this we have reduced to the case that
\[
    x_{j,k} = \begin{cases}
                &a_j,\quad j\leq k,\\
                &b_k,\quad j>k.
              \end{cases}
\]
So we assume that $x$ has this form. 

Recall that $a_j\to1$ as $j\to\infty$ and $b_k\to 0$ as $k\to\infty.$ Fix $N\in\mathbb{N}$ such that $|a_j-1|<\varepsilon$ and $|b_k|<\varepsilon$ as $j,k>N.$ Now, set $\sigma(j)=N+j$ and $\pi(k)=N+k.$ If $j\leq k,$ then $N+j\leq N+k$ and
$$|x_{\sigma(j),\pi(k)}-1|=|a_{N+j}-1|<\varepsilon.$$
If $j>k,$ then $N+j>N+k$ and
$$|x_{\sigma(j),\pi(k)}|=|b_{N+k}|<\varepsilon.$$
This completes the proof.
\end{proof}

%
%
%
% define the sequence $\{(\sigma(j),\pi(j))\}_{j\geq 0}$ inductively by:
% \begin{align*}
%     \sigma(j) &= \min\{m> \sigma(j-1)\;:\; \max_{k\leq j-1}|x_{m,\pi(k)}-1|< \varepsilon\},\\
%     \pi(j) &= \min\{m> \pi(j-1)\;:\; \max_{\ell\leq j}|x_{\sigma(\ell),m}| < \varepsilon\}.
% \end{align*}

\begin{lemma}\label{non_differentiable_implies_triangle}
Let $(\Ec,\|\cdot\|_{\Ec})$ be a quasi-Banach ideal with the Fatou property. If there exists a $\Ec$-Lipschitz function which fails to be everywhere differentiable, then $\|\Gamma\|_{M(\Ec)}<\infty.$
\end{lemma}
\begin{proof} Let $f$ be an $\Ec$-Lipschitz function which is not differentiable at some point. Without loss of generality, $f(0)=0$ and $f$ is not differentiable at $0.$ Since $f$ is Lipschitz continuous, the set
    \[
        \Big\{\frac{f(h)}{h}\Big\}_{h\in [-1,1]}
    \]
    is bounded. Again without loss of generality, we can choose sequences $\{\lambda_j\}_{j=0}^\infty$ and $\{\mu_k\}_{k=0}^\infty$ converging to zero such that
    \[
        \lim_{j\to\infty} \frac{f(\lambda_j)}{\lambda_j} =1,\quad \lim_{k\to\infty} \frac{f(\mu_k)}{\mu_k} =0.
    \]
    Refining the sequences if necessary, we may assume that $\lambda_j\neq \mu_k$ for all $j\neq k.$
    In particular, we have
    \[
        \lim_{j\to\infty}\lim_{k\to\infty} \frac{f(\lambda_j)-f(\mu_k)}{\lambda_j-\mu_k} = 1,\quad \lim_{k\to\infty}\lim_{j\to\infty} \frac{f(\lambda_j)-f(\mu_k)}{\lambda_j-\mu_k} = 0.
    \]
    By Lemma \ref{injection_lemma}, for all $n\geq 1$ there exist injective functions $\sigma_n,\pi_n:\mathbb{Z}_+\to\mathbb{Z}_+$  such that
$$\Big|\frac{f(\lambda_{\sigma_n(j)})-f(\mu_{\pi_n(k)})}{\lambda_{\sigma_n(j)}-\mu_{\pi_n(k)}}-\Gamma(j,k)\Big|<\frac1n,\quad j,k\geq 0.$$
    Let $X \in \Ec$ be represented by a finite matrix $\{X_{j,k}\}_{j,k\geq 0}.$ By Lemma \ref{Lipschitz_is_the_same_as_Schur_condition}, there exists a constant $C_f$ such that
    \[
        \sup_{n\geq 0} \Big\|\Big\{\frac{f(\lambda_{\sigma_n(j)})-f(\mu_{\pi_n(k)})}{\lambda_{\sigma_n(j)}-\mu_{\pi_n(k)}}X_{j,k}\Big\}_{j,k\geq 0}\Big\|_{\Ec} \leq C_f\|X\|_{\Ec}.
    \]
    Since $X$ is is represented by a finite matrix, we have (c.f. Remark \ref{componentwise_remark})
    \[
        \Big\|\Big\{\Gamma(j,k)X_{j,k}\Big\}_{j,k\geq 0}\Big\|_{\Ec}  = \lim_{n\to\infty} \Big\|\Big\{\frac{f(\lambda_{\sigma_n(j)})-f(\mu_{\pi_n(k)})}{\lambda_{\sigma_n(j)}-\mu_{\pi_n(k)}}X_{j,k}\Big\}_{j,k\geq 0}\Big\|_{\Ec} \leq C_f\|X\|_{\Ec}.
    \]
    That is, $\Gamma$ is a weak $\Ec$-Schur multiplier. By Lemma \ref{schur_definitions_are_equivalent_for_fatou}, it follows that $\|\Gamma\|_{M(\Ec)} < \infty.$

\end{proof}

% As a partial converse of Lemma \ref{non_differentiable_implies_triangle}, we have the following:
% \begin{lemma}
%     If $\Ec$ has the Fatou property and $\|\Gamma\|_{M(\Ec)}<\infty,$ then the absolute value function is $\Ec$-Lipschitz.
% \end{lemma}
% \begin{proof}
%     By Lemma \ref{schur_definitions_are_equivalent_for_fatou}, $\Gamma$ is an $\Ec$-Schur multiplier.
%     By Lemma \ref{triangular_truncation_bounded_implies_fully_symmetric}, $\Ec$ is closed under submajorisation and therefore up to an equivalent norm $\Ec$ is fully symmetric (Lemma \ref{closed_under_submajorisation_implies_fully_symmetric}).
%
%     Set $f(t) = |t|.$ We have
%    \[
%         f^{[1]}(\lambda,\mu) = \frac{|\lambda|-|\mu|}{\lambda-\mu} = \frac{|\lambda|+|\mu|}{|\lambda+\mu|}\cdot \frac{|\lambda+\mu|}{\lambda+\mu}
%    \]
%    The $\Ec$-Schur boundedness of $\Gamma$ implies the $\Ec$-Schur boundedness of the function
%    \[
%         (\lambda,\mu)\mapsto \frac{|\lambda+\mu|}{\lambda+\mu}
%    \]
%    On the other hand, the function
%    \[
%         (\lambda,\mu) \mapsto \frac{|\lambda+\mu|}{|\lambda|+|\mu|}
%    \]
%    is a $\Bc$-Schur multiplier by Lemma \ref{triangle_matrix_is_bounded}, and hence a $\Cc_1$-Schur multiplier. Since the fully symmetric ideals are the interpolation spaces between $\Cc_1$ and $\Bc,$ it follows that $\Gamma$ is also $\Ec$-Schur bounded.
%
%    It follows that the divided difference of the absolute value function is an $\Ec$-Schur multiplier, and hence that the absolute value function is $\Ec$-Lipschitz.
% \end{proof}

\begin{remark}
A combination of Theorems 3.4 and 4.4 of \cite{DDdPS} asserts that for a \emph{Banach} ideal $(\Ec,\|\cdot\|_{\Ec}),$ the following are equivalent:
\begin{enumerate}
    \item{} The triangular truncation is a bounded $\Ec$-Schur multiplier
    \item{} The absolute value function $f(x) = |x|$ is $\Ec$-Lipschitz
    \item{} $\Ec$ has nontrivial Boyd indices.
\end{enumerate}
Lemma \ref{non_differentiable_implies_triangle} extends that result to the quasi-Banach case, but the more important point is that the statement ``$f(x)=|x|$ is $\Ec$-Lipschitz" turns out to be equivalent to ``there exists a non-differentiable $\Ec$-Lipschitz function" at least when $\Ec$ has the Fatou property.
\end{remark}

Immediately from Lemmas \ref{triangular_truncation_bounded_implies_fully_symmetric} and \ref{non_differentiable_implies_triangle}, we deduce the following:
\begin{corollary}\label{not_fully_symm_implies_differentiable}
    If $(\Ec,\|\cdot\|_{\Ec})$ is a quasi-Banach ideal with the Fatou property which is not closed under triangular truncation, then every $\Ec$-Lipschitz function is differentiable.
\end{corollary}

\begin{lemma}\label{derivative_vanishes_at_infinity}
    Let $\Ec$ be a quasi-Banach ideal with the Fatou property which is not closed under submajorisation. If $f$ is a bounded $\Ec$-Lipschitz function, then
    \[
        \lim_{\lambda\to\pm\infty} f'(\lambda)=0.
    \]
\end{lemma}
\begin{proof}
    By Theorem \ref{not_fully_symm_implies_differentiable}, $f$ is differentiable, and so $f'$ is at least meaningful. By symmetry we only need to prove that $\lim_{\lambda\to\infty} f'(\lambda)=0.$

    Suppose the contrary, and there exists an $\Ec$-Lipschitz function such that $f'(\lambda)$ does not vanish as $\lambda\to\infty.$ Without loss of generality, we may assume that
    \[
        \limsup_{\lambda\to\infty} |f'(\lambda)| >0.
    \]
    It follows that we can find increasing sequences $\{\lambda_j\}_{j=0}^\infty$ and $\{\mu_j\}_{j=0}^\infty$ such that $\lambda_j,\mu_j\uparrow\infty,$ $\lambda_j-\mu_j\to0,$
$$|\lambda_j-\mu_k|\geq |j-k|,\quad j,k\geq0,$$
$$\lim_{j\to\infty}\frac{f(\lambda_j)-f(\mu_j)}{\lambda_j-\mu_j}=c\neq0.$$

    If $f$ were $\Ec$-Lipschitz, then by Lemma \ref{Lipschitz_is_the_same_as_Schur_condition} there exists a constant $C_f$ such that
    \begin{equation}\label{lowner_in_derivative_argument}
        \left\|\left\{\frac{f(\lambda_{n(j+1)})-f(\mu_{n(k+1)})}{\lambda_{n(j+1)}-\mu_{n(k+1)}}X_{j,k}\right\}_{j,k\geq 0}\right\|_{\Ec} \leq C_f\|X\|_{\Ec}
    \end{equation}
    for every $n\geq 1$ and for every $X$ with finite matrix representation.

We have (c.f. Remark \ref{componentwise_remark})
    \[
        \Big\|\Big\{c\chi_{\Delta}(j,k)X_{j,k}\Big\}_{j,k\geq 0}\Big\|_{\Ec}  = \lim_{n\to\infty} \Big\|\Big\{\frac{f(\lambda_{n(j+1)})-f(\mu_{n(k+1)})}{\lambda_{n(j+1)}-\mu_{n(k+1)}}X_{j,k}\Big\}_{j,k\geq 0}\Big\|_{\Ec} \leq \|f^{[1]}\|_{\Ec}\|X\|_{\Ec}.
    \]
    That is, $\|\chi_{\Delta}\|_{m(\Ec)}<\infty.$ Since $\Ec$ has the Fatou property, it follows that the diagonal is an $\Ec$-Schur multiplier, and therefore by By Theorem \ref{pattern_theorem} it follows that $\Ec$ is closed under submajorisation, which is a contradiction.
\end{proof}

\begin{remark}
    Examination of the proof of Lemma \ref{derivative_vanishes_at_infinity} shows that the assumption that $f$ be bounded can be replaced with
    \[
        \lim_{N\to\infty} \sup_{|\lambda-\mu|>N} \left|\frac{f(\lambda)-f(\mu)}{\lambda-\mu}\right| = 0.
    \]
\end{remark}

% The idea of the proof is essentially the same as in \cite{MS}, using the following lemma.
% \begin{lemma}\label{gamma_bounds_diag}
%     For $0<\varepsilon<1,$ let $\Gamma$ be the matrix
%     \[
%         \Gamma = \{(\varepsilon+j-k)^{-1}\}_{j,k\geq 0}.
%     \]
%     Let $\Ec$ be a quasi-Banach ideal of $\Bc.$ We have
%     \[
%         \|\mathrm{diag}\|_{M(\Ec)} \leq \|\Gamma\|_{M(\Ec)}
%     \]
%     where $M(\Ec)$ is the $\Ec$-Schur multiplier norm.
% \end{lemma}
% \begin{proof}
%     Let $X \in E,$ and let $X^N$ be the $N$-fold dilation of $X$ given by
%     \[
%         X^N = X\otimes e_{0,0}
%     \]
%     where $e_{0,0}$ is the $N\times N$ matrix with $1$ in the top left entry and zeros elsewhere. We have
%     \[
%         \|\Gamma \schur X^N\|_{\Ec} = \|\{\frac{X_{j,k}}{\varepsilon+N(j-k)}\}\|_{\Ec}.
%     \]
%     Hence, for any $N,$
%     \[
%         \|\{\frac{X_{j,k}}{\varepsilon+N(j-k)}\}\|_{\Ec} \leq \|\Gamma\|_{m(\Ec)}\|X\|_{\Ec}.
%     \]
%     Taking the limit as $N\to\infty$ yields
%     \[
%         \|\mathrm{diag}(X)\|_{\Ec} \leq \|\Gamma\|_{m(\Ec)}\|X\|_{\Ec}.
%     \]
% \end{proof}

\begin{proof}[Proof of Theorem \ref{periodic_functions_theorem}]
    Since, e.g., $f(x) = \exp(ix)$ is an operator Lipschitz function which does not have $f'(x)\to 0$ as $x\to\infty,$ and hence \eqref{has_not_OL_periodic} implies \eqref{has_vanishing_derivative_periodic}.

    Lemma \ref{derivative_vanishes_at_infinity} already proves that \eqref{has_vanishing_derivative_periodic} implies \eqref{is_not_closed_periodic}.

    Finally, we prove that \eqref{is_not_closed_periodic} implies \eqref{has_not_OL_periodic}. Assume that $\Ec$ is closed under submajorisation. By Lemma \ref{closed_under_submajorisation_implies_fully_symmetric}, we may assume loss of generality that $\Ec$ is fully symmetric. If $f$ is an operator-Lipschitz function, by Theorem \ref{Lowner_schur_theorem}, $f^{[1]}$ is a $\Bc$-Schur multiplier, and hence by duality is a $\Cc_1$-Schur multiplier. Since $\Cc_1$ has the Fatou property, we can deduce that $f$ is also $\Cc_1$-Lipschitz. Let $X = X_0+X_1,$ where $X_0\in \Bc$ and $X_1\in \Cc_1.$ The $K$-functional for $f(A+X)-f(X)$ for the couple $(\Bc,\Cc_1)$ is bounded by
    \begin{align*}
        K(t,f(A+X)-f(A);\Bc,\Cc_1) &\leq \|f(A+X_0+X_1)-f(A+X_0)\|_{\Cc_1}+t\|f(A+X_0)-f(A)\|_{\Bc}\\
                                   &\leq C_f(\|X_1\|_{\Cc_1}+t\|X_0\|_{\Bc}).
    \end{align*}
    Taking the infimum over all $X_0,X_1$ such that $X=X_0+X_1,$ it follows that
    \[
        K(t,f(A+X)-f(A);\Bc,\Cc_1) \leq C_fK(t,X;\Bc,\Cc_1),\quad t>0.
    \]
    Therefore, by the Holmstedt formula,
    \[
        f(A+X)-f(A)\prec\prec C_fX.
    \]
    Since $\Ec$ is fully symmetric it follows that $f$ is $\Ec$-Lipschitz.
\end{proof}

\appendix

\section{On the L\"owner condition for $\Ec$-Lipschitz functions}\label{lowner_appendix}
For $\lambda,\mu\in \Rl$ and a function $f:\Rl\to\Cplx,$ define
\[
    f^{[1]}(\lambda,\mu) := \frac{f(\lambda)-f(\mu)}{\lambda-\mu},\quad \lambda\neq\mu
\]
and $f^{[1]}(\lambda,\lambda)=0.$

The following theorem is folklore:
\begin{theorem}\label{Lowner_schur_theorem}
    Let $f:\Rl\to \Cplx$ be Borel measurable, and let $(\Ec,\|\cdot\|_{\Ec})$ be a quasi-Banach ideal of $\Bc.$ The following are equivalent:
    \begin{enumerate}[{\rm (i)}]
        \item{}\label{E_Lipschitzness} There exists a constant $C_f$ such that for all self-adjoint $A,B\in \Bc$ such that $A-B \in \Ec,$ we have
        \[
            \|f(A)-f(B)\|_{\Ec} \leq C_f\|A-B\|_{\Ec}.
        \]
        \item{}\label{E_commutator} There exists a constant $C_f$ such that for all self-adjoint $A,X \in \Bc$ such that $AX-XA \in \Ec$ we have
        \[
            \|f(A)X-Xf(A)\|_{\Ec} \leq C_f\|AX-XA\|_{\Ec}.
        \]
        \item{}\label{E_quasi_commutator} There exists a constant $C_f$ such that for all self-adjoint $A,B \in \Bc$ and all $X\in \Bc$ such that $AX-XB\in \Ec$ we have
        \[
            \|f(A)X-Xf(B)\|_{\Ec} \leq C_f\|AX-XB\|_{\Ec}.
        \]
    \end{enumerate}
    If any of the above are true, then
    \begin{equation}\label{Lowner_boundedness}
        \sup_{\{\lambda_j\},\{\mu_k\}} \|\{f^{[1]}(\lambda_j,\mu_k)\}_{j,k\geq 0}\|_{M(\Ec)} < \infty.
    \end{equation}
    where the supremum is taken over all sequences $\{\lambda_j,\mu_k\}_{j,k=0}^\infty$ which are non-intersecting.
    Conversely, if $\Ec$ has the Fatou property then \eqref{Lowner_boundedness} implies \eqref{E_Lipschitzness}, \eqref{E_commutator} and \eqref{E_quasi_commutator}.
\end{theorem}
For the case when $\Ec$ is a Banach ideal, Theorem \ref{Lowner_schur_theorem} is \cite[Section 5]{KissinShulman}. The general case is not significantly different but we take this opportunity to fill this gap in the literature.

\begin{remark}
    The Fatou property is used in the proof that \eqref{Lowner_boundedness} implies \eqref{E_Lipschitzness}, \eqref{E_commutator} and \eqref{E_quasi_commutator} of Theorem \ref{Lowner_schur_theorem} only in passing from the finite to infinite dimensional case.

    If in \eqref{E_Lipschitzness}, \eqref{E_commutator} and \eqref{E_quasi_commutator} we placed the further restriction that all the operators $A,X,B$ have finite dimensional cokernel, then this assumption would be unnecessary.
\end{remark}

For the general quasi-Banach case, we will use the Aoki-Rolewicz theorem \cite[Chapter 1, Theorem 1.3]{KPR}, which asserts that if $(\Ec,\|\cdot\|_{\Ec})$ is a quasi-Banach ideal then the norm of $\Ec$ is $p$ convex for some $0< p\leq 1$ in the sense that
\begin{equation}\label{p_triangle_inequality}
    \|A_1+\cdots+A_n\|_{\Ec} \leq C_{\Ec}\left(\|A_1\|_{\Ec}^p+\cdots+\|A_n\|_{\Ec}^p\right)^{1/p}.
\end{equation}
Moreover, the topology of $\Ec$ is $p$-Banach in the sense that if $\sum_{k=1}^\infty \|A_k\|_{\Ec}^p<\infty$ then $\sum_{k=1}^\infty A_k$ converges in $\Ec.$

\begin{lemma}\label{differentiate_the_exponential}
    Let $(\Ec,\|\cdot\|_{\Ec})$ be a quasi-Banach ideal. Let $A,X \in \Bc$ be self-adjoint
    and such that $AX-XA \in \Ec$.
    For all $t \in \Rl$ we have $e^{itX}A-Ae^{itX} \in \Ec$ and there exist constants $c_{\Ec}, C_{\Ec}$ such that
    \begin{equation*}
        \limsup_{t\to 0} \frac{1}{t}\|e^{itX}A-Ae^{itX}\|_\Ec \leq C_{\Ec}\|XA-AX\|_{\Ec}
    \end{equation*}
    and
    \[
        c_{\Ec}\|XA-AX\|_{\Ec} \leq \liminf_{t\to 0} \frac{1}{t}\|e^{itX}A-Ae^{itX}\|_{\Ec}.
    \]
\end{lemma}
\begin{proof}
    Let $0<p\leq 1$ be such that \eqref{p_triangle_inequality} holds and $\Ec$ is $p$-Banach.
    Since $A$ and $X$ are bounded, the series
    \begin{equation*}
        e^{itX}A-Ae^{itX} = [itX,A]+\sum_{k=2}^\infty \frac{(it)^k}{k!}(X^kA-AX^k)
    \end{equation*}
    converges in the norm of $\Bc.$
    By the Leibniz rule, for all $k\geq 1$ we have
    \begin{equation*}
        \|X^kA-AX^k\|_{\Ec} \leq C_{\Ec}k^{1/p}\|X\|_\infty^{k-1}\|XA-AX\|_{\Ec}.
    \end{equation*}
    Hence the series $\sum_{k=2}^\infty t^{kp}(k!)^{-p}\|X^kA-AX^k\|_{\Ec}^p$ is finite. In particular,
    \[
        e^{itX}A-Ae^{itX} \in \Ec.
    \]
    Moreover,
    \begin{equation*}
        \left\|\frac{1}{it}(e^{itX}A-Ae^{itX}) - (XA-AX)\right\|_\Ec \leq C_{\Ec}\left(\sum_{k=2}^\infty \frac{t^{(k-1)p}}{(k!)^p}(k+1)\|X\|_{\infty}^{pk}\|XA-AX\|_{\Ec}^p\right)^{1/p}.
    \end{equation*}
    Therefore
    \begin{equation*}
        \lim_{t\to 0} \|\frac{1}{it}(e^{itX}A-Ae^{itX})-(XA-AX)\|_{\Ec} = 0.
    \end{equation*}
    This implies the two inequalities, by the quasi-triangle inequality for $\Ec.$
\end{proof}

\begin{proof}[Proof of Theorem \ref{Lowner_schur_theorem}]
    First we prove that \eqref{E_Lipschitzness}, \eqref{E_commutator} and \eqref{E_quasi_commutator} are equivalent.

    \noindent \eqref{E_Lipschitzness}$\Rightarrow$\eqref{E_commutator}
    Let $f$ satisfy \eqref{E_Lipschitzness} and $A,X\in \Bc$ are self-adjoint such that that $AX-XA\in \Ec.$ By Lemma \ref{differentiate_the_exponential}, we have
    \[
        A-e^{itX}Ae^{-itX} \in \Ec.
    \]
    and therefore
    \[
        \|f(A)-e^{itX}f(A)e^{-itX}\|_{\Ec} \leq C_f\|A-e^{itX}Ae^{-itX}\|_{\Ec}.
    \]
    That is,
    \[
        \frac{1}{t}\|e^{itX}f(A)-f(A)e^{itX}\|_{\Ec} \leq C_f\frac{1}{t}\|e^{itX}A-Ae^{itX}\|_{\Ec},\quad t \geq 0.
    \]
    By Lemma \ref{differentiate_the_exponential}, we deduce that $f$ satisfies \eqref{E_commutator}.\\
    \noindent \eqref{E_commutator}$\Rightarrow$\eqref{E_quasi_commutator}
    Now assume that $A,B\in \Bc$ are self-adjoint, $X \in \Bc$ and $AX-XB\in \Ec$ and $f$ satisfies \eqref{E_commutator}. Using an isomorphism $H\approx H\oplus H,$ let $\widetilde{A},\widetilde{X}\in \Bc$ be defined by
    \[
        \widetilde{A} = \begin{pmatrix} A & 0 \\ 0 & B\end{pmatrix},\quad \widetilde{X} = \begin{pmatrix} 0 & X \\ X^* & 0 \end{pmatrix}.
    \]
    We have
    \begin{align*}
        f(\widetilde{A})\widetilde{X}-\widetilde{X}f(\widetilde{A}) &= \begin{pmatrix} 0 & f(A)X-Xf(B)\\ f(B)X^*-X^*f(A) & 0 \end{pmatrix},\\
        \widetilde{A}\widetilde{X}-\widetilde{X}\widetilde{A} &= \begin{pmatrix} 0 & AX-XB\\ BX^*-X^*A\end{pmatrix}.
    \end{align*}
    Therefore
    \[
        \|f(A)X-Xf(B)\|_{\Ec} \lesssim \|f(\widetilde{A})\widetilde{X}-\widetilde{X}f(\widetilde{A})\|_{\Ec} \lesssim C_f\|\widetilde{A}\widetilde{X}-\widetilde{X}\widetilde{A}\|_{\Ec} \lesssim C_f\|AX-XB\|_{\Ec}.
    \]
    That is, $f$ satisfies \eqref{E_quasi_commutator}.

    \noindent The implication \eqref{E_quasi_commutator}$\Rightarrow$\eqref{E_Lipschitzness} is trivial.

    Next, we show that \eqref{E_quasi_commutator} implies \eqref{Lowner_boundedness}.
    Let $\{\lambda_j\}_{j= 0}^\infty$ and $\{\mu_k\}_{k=0}^\infty$ be disjoint real sequences, and let $X = \{X_{j,k}\}_{j,k=0}^\infty$ be the matrix representation of an operator in $\Bc.$
    Let
    \[
        A = \{\lambda_j\delta_{j,k}\}_{j,k=0}^\infty,\;\quad B = \{\mu_j\delta_{j,k}\}_{j,k=0}^\infty.
    \]
    Then
    \[
        f(A)X-Xf(A) = \{(f(\lambda_j)-f(\mu_k))X_{j,k}\}_{j,k=0}^\infty,\quad AX-XB = \{(\lambda_j-\mu_k)X_{j,k}\}_{j,k=0}^\infty.
    \]
    Let $Z= \{Z_{j,k}\}_{j,k=0}^\infty \in \Ec$ be a finite matrix, and define
    \[
        X_{j,k} :=     (\lambda_j-\mu_k)^{-1}Z_{j,k},\quad j,k\geq 0
    \]
    so that
    \[
        AX-XB = Z
    \]
    It follows that
    \[
            \|\{f^{[1]}(\lambda_j,\mu_k)Z_{j,k}\}_{j,k=0}^\infty\|_{\Ec} = \|f(A)X-Xf(B)\|_{\Ec} \leq C_f\|AX-XB\|_{\Ec} = C_f\|Z\|_{\Ec}.
    \]
    Since the sequences are arbitrary disjoint sequences, we conclude that $\|f^{[1]}\|_{M(\Ec)}\leq 1.$

    Finally, we assume that $\Ec$ has the Fatou property, and prove that \eqref{Lowner_boundedness} implies \eqref{E_Lipschitzness}.
    Initially suppose that $A,B \in \Ec$ and have disjoint non-zero spectra, i.e.
    \[
        Ax = \sum_{j=0}^\infty \lambda_j\xi_k\langle \xi_k,x\rangle,\quad B = \sum_{k=0}^\infty \mu_k \eta_k\langle \eta_k,x\rangle,\quad x \in H.
    \]
    where $\{\xi_j\}$ and $\{\eta_k\}$ are orthonormal bases sets of mutually orthogonal rank $1$ projections, and $\lambda_j\neq \mu_k$ for all $j,k.$
    Then
    \[
        \langle \xi_j,(f(A)-f(B))\eta_k\rangle = (f(\lambda_j)-f(\mu_k))\langle \xi_j,\eta_k\rangle = f^{[1]}(\lambda_j,\mu_k)\langle \xi_j,(A-B)\eta_k\rangle.
    \]
    Therefore, assuming that $f$ satisfies \eqref{Lowner_boundedness}, it follows that
    \[
        \|f(A)-f(B)\|_{\Ec} \leq \|f^{[1]}\|_{M(\Ec)}\|A-B\|_{\Ec}
    \]
    Now consider the general case $A,B \in \Ec$ with possibly non-disjoint spectra.
    For any $n\geq 0,$ we choose $B_{n}\in \Ec$ with nonzero spectrum disjoint from that of $A$ and $B$ such that $\|B_{n}-B\|_{\Ec}\leq \frac1n.$ By the continuity of functional calculus, we have
    \[
        \|f(A)-f(B)\|_{\Ec} \lesssim \|f^{[1]}\|_{M(\Ec)}(\|A-B\|_{\Ec}+\|B-B_{n}\|_{\Ec}),\quad n\geq 1.
    \]
    Taking the limit, we conclude that
    \[
        \|f(A)-f(B)\|_{\Ec} \lesssim \|f^{[1]}\|_{M(\Ec)}\|A-B\|_{\Ec},\quad A=A^*,B=B^* \in \Ec.
    \]
    In general, let $A,B \in \Bc$ be self-adjoint such that $A-B \in \Ec.$ Let $\{p_n\}_{n\geq 0}$ be a sequence of finite-rank projections strongly converging to $1.$ Then $p_nAp_n,p_nBp_n\in \Ec$ and for any finite-rank contractions $q,r$ we have
    \[
        \|q(f(p_nAp_n)-f(p_nBp_n))r\|_{\Ec} \lesssim \|f^{[1]}\|_{M(\Ec)}\|p_n(A-B)p_n\|_{\Ec} \leq \|f^{[1]}\|_{M(\Ec)}\|A-B\|_{\Ec}.
    \]
    By the continuity of functional calculus, $f(p_nAp_n)\to f(A)$ and $f(p_nBp_n)\to f(B)$ in the strong operator topology, and since $q,r$ have finite rank we have $qf(p_nAp_n)r\to qf(A)r$ in the topology of $\Ec,$ and similarly $qf(p_nBp_n)r\to qf(B)r$ in $\Ec.$ Taking the limit as $n\to\infty,$ we conclude that
    \[
        \|q(f(A)-f(B))r\|_{\Ec}\lesssim \|f^{[1]}\|_{M(\Ec)}\|A-B\|_{\Ec}.
    \]
    In particular, for any finite-rank projection $q$ we have
    \[
        \||f(A)-f(B)|^{1/2}q|f(A)-f(B)|^{1/2}\|_{\Ec} = \|q|f(A)-f(B)|q\|_{\Ec} \lesssim \|f^{[1]}\|_{M(\Ec)} \|A-B\|_{\Ec}.
    \]
    Since $\Ec$ has the Fatou property and $q$ is arbitrary, we conclude that
    \[
        \|f(A)-f(B)\|_{\Ec} \lesssim \|f^{[1]}\|_{M(\Ec)}\|A-B\|_{\Ec}.
    \]

\end{proof}

\section*{Acknowledgments}
E.M. acknowledges the support of the Alexander von Humboldt foundation, Universit\"at Bonn and Harbin Institute of Technology. No large language model was involved at any stage of the preparation of this paper. There are no conflicts of interest to declare. Data sharing is not applicable.

\end{document}